\documentclass[11pt]{article}
\usepackage{amssymb}
\usepackage{amsmath}
\usepackage{latexsym}
\usepackage{mathrsfs}
\usepackage{algorithm}
\usepackage{algorithmic}
\usepackage{verbatim}
\usepackage{url}

\newtheorem{thm}{Theorem}

\newtheorem{prop}{Proposition}

\newtheorem{fact}{Fact}
\newtheorem{conj}{Conjecture}

\newcommand{\vertiii}[1]{{\left\vert\kern-0.2ex\left\vert\kern-0.2ex\left\vert #1 \right\vert\kern-0.2ex\right\vert\kern-0.2ex\right\vert}}

\newenvironment{proof}{{\bf Proof\,\,}}{\endproof\par}

\newcounter{spb}
\newcommand{\subpb}{(\alph{spb}) \addtocounter{spb}{1}}

\newcommand{\resetspb}{\setcounter{spb}{1}}

\def \openbox{$\sqcup\llap{$\sqcap$}$}
\def \endproof{\enskip \null \nobreak \hfill \openbox \par}

\newcommand{\limto}{\rightarrow}
\newcommand{\del}{\partial}
\newcommand{\bz}{\mathbf 0}

\newcommand{\R}{\mathbb R}

\begin{document}
\title{A Perturbation Inequality for the Schatten $p$--Quasi--Norm \\ and Its Applications in Low--Rank Matrix Recovery\thanks{This research is supported in part by the Hong Kong Research Grants Council (RGC) General Research Fund (GRF) Project CUHK 416413, and in part by a gift grant from Microsoft Research Asia.}}
\author{Man--Chung Yue\thanks{Department of Systems Engineering and Engineering Management, The Chinese University of Hong Kong, Shatin, N.~T., Hong Kong.  E--mail: {\tt mcyue@se.cuhk.edu.hk}} \and Anthony Man--Cho So\thanks{Department of Systems Engineering and Engineering Management, and, by courtesy, CUHK--BGI Innovation Institute of Trans--omics, The Chinese University of Hong Kong, Shatin, N.~T., Hong Kong.  E--mail: {\tt manchoso@se.cuhk.edu.hk}}}
\date{\today}
\maketitle

\begin{abstract}
In this paper, we establish the following perturbation result concerning the singular values of a matrix: Let $A,B \in \R^{m\times n}$ be given matrices, and let $f:\R_+\limto\R_+$ be a concave function satisfying $f(0)=0$.  Then, we have
$$ \sum_{i=1}^{\min\{m,n\}} \big| f(\sigma_i(A)) - f(\sigma_i(B)) \big| \le \sum_{i=1}^{\min\{m,n\}} f(\sigma_i(A-B)), $$
where $\sigma_i(\cdot)$ denotes the $i$--th largest singular value of a matrix.  This answers an open question that is of interest to both the compressive sensing and linear algebra communities.  In particular, by taking $f(\cdot)=(\cdot)^p$ for any $p \in (0,1]$, we obtain a perturbation inequality for the so--called Schatten $p$--quasi--norm, which allows us to confirm the validity of a number of previously conjectured conditions for the recovery of low--rank matrices via the popular Schatten $p$--quasi--norm heuristic.  We believe that our result will find further applications, especially in the study of low--rank matrix recovery.
\end{abstract}

{\bf Keywords:} Singular value perturbation inequality; Schatten quasi--norm; Low--rank matrix recovery

\section{Introduction} \label{sec:intro}
The problem of low--rank matrix recovery, with its many applications in computer vision~\cite{CS04,JLSX10}, trace regression~\cite{NW11,KLT11}, network localization~\cite{JM13,JSZ+13}, etc., has been attracting intense research interest in recent years.  In a basic version of the problem, the goal is to reconstruct a low--rank matrix from a set of possibly noisy linear measurements.  To achieve this, one immediate idea is to formulate the recovery problem as a rank minimization problem:
\begin{equation} \label{eq:RMP}
\begin{array}{c@{\quad}l}
\mbox{minimize} & \mbox{rank}(X) \\
\noalign{\smallskip}
\mbox{subject to} & \| \mathcal{A}(X) - y \|_2 \le \eta, \,\,\, X \in \R^{m\times n},
\end{array}
\end{equation}
where the linear measurement map $\mathcal{A}:\R^{m \times n} \limto \R^l$, the vector of measurements $y \in \R^l$, and the noise level $\eta \ge 0$ are given.  However, Problem (\ref{eq:RMP}) is NP--hard in general, as it includes the NP--hard vector cardinality minimization problem~\cite{N95} as a special case.  Moreover, since the rank function is discontinuous, Problem (\ref{eq:RMP}) can be challenging from a computational point--of--view.  To circumvent this intractability, a popular approach is to replace the objective of (\ref{eq:RMP}) with the so--called Schatten (quasi)--norm of $X$.  Specifically, given a matrix $X \in \R^{m\times n}$ and a number $p\in(0,1]$, let $\sigma_i(X)$ denote the $i$--th largest singular value of $X$ and define the {\it Schatten $p$--quasi--norm} of $X$ by
$$ \|X\|_p = \left( \sum_{i=1}^{\min\{m,n\}} \sigma_i^p(X) \right)^{1/p}. $$
One can then consider the following {\it Schatten $p$--quasi--norm heuristic} for low--rank matrix recovery:
\begin{equation} \label{eq:schatten}
\begin{array}{c@{\quad}l}
\mbox{minimize} & \|X\|_p^p \\
\noalign{\smallskip}
\mbox{subject to} & \| \mathcal{A}(X) - y \|_2 \le \eta, \,\,\, X \in \R^{m\times n}.
\end{array}
\end{equation}
Note that the function $X \mapsto \|X\|_p^p$ is continuous for each $p\in(0,1]$.  Thus, algorithmic techniques for continuous optimization can be used to tackle Problem (\ref{eq:schatten}).  The Schatten quasi--norm heuristic is motivated by the observation that $\|X\|_p^p \limto \mbox{rank}(X)$ as $p \searrow 0$.  In particular, when $p=1$, the function $X \mapsto \|X\|_1$ defines a norm---known as the {\it nuclear norm}---on the set of $m\times n$ matrices, and we obtain the well--known {\it nuclear norm heuristic}~\cite{FHB01}.  In this case, Problem (\ref{eq:schatten}) is a convex optimization problem that can be solved efficiently by various algorithms; see, e.g.,~\cite{HZY+13} and the references therein.  On the other hand, when $p \in (0,1)$, the function $X \mapsto \|X\|_p$ only defines a quasi--norm.  In this case, Problem (\ref{eq:schatten}) is a non--convex optimization problem and is NP--hard in general; cf.~\cite{GJY11}.  Nevertheless, a number of numerical algorithms implementing the Schatten $p$--quasi--norm heuristic (where $p \in (0,1)$) have been developed (see, e.g.,~\cite{MS12,NHD12,JSZ+13,LXY13} and the references therein), and they generally have better empirical recovery performance than the (convex) nuclear norm heuristic.

From a theoretical perspective, a natural and fundamental question concerning the aforementioned heuristics is about their recovery properties.  Roughly speaking, this entails determining the conditions under which a given heuristic can recover, either exactly or approximately, a solution to Problem (\ref{eq:RMP}).  A first study in this direction was done by Recht, Fazel and Parrilo~\cite{RFP10}, who showed that techniques used to analyze the $\ell_1$ heuristic for sparse vector recovery (see~\cite{S12} for an overview and further pointers to the literature) can be extended to analyze the nuclear norm heuristic.  Since then, recovery conditions based on the restricted isometry property (RIP) and various nullspace properties have been established for the nuclear norm heuristic; see, e.g.,~\cite{OMFH11,CR13,CZ13,JKN14} for some recent results.  In fact, many recovery conditions for the nuclear norm heuristic can be derived in a rather simple fashion from their counterparts for the $\ell_1$ heuristic by utilizing a perturbation inequality for the nuclear norm~\cite{OMFH11}.

Compared with the nuclear norm heuristic, recovery properties of the Schatten $p$--quasi--norm heuristic are much less understood, even though the corresponding heuristic for sparse vector recovery, namely the $\ell_p$ heuristic with $p \in (0,1)$, has been extensively studied; see, e.g.,~\cite{WXT11,WC13} and the references therein.  As first pointed out in~\cite{OMFH11} and later further elaborated in~\cite{LLLW12}, the difficulty seems to center around the following question, which concerns the validity of certain perturbation inequality for the Schatten $p$--quasi--norm:

\medskip
\noindent{\bf Question (Q)} {\it Given a number $p\in (0,1)$ and matrices $A,B \in \R^{m\times n}$, does the inequality 
\begin{equation} \label{eq:key-ineq}
\sum_{i=1}^{\min\{m,n\}} \left| \sigma_i^p(A) - \sigma_i^p(B) \right| \le \sum_{i=1}^{\min\{m,n\}} \sigma_i^p(A-B)
\end{equation}
hold?}

\medskip
\noindent Indeed, assuming the validity of (\ref{eq:key-ineq}), one can establish a {\it necessary and sufficient} nullspace--based condition for the recovery of low--rank matrices by the Schatten $p$--quasi--norm heuristic~\cite{OMFH11}.  This, coupled with the arguments in~\cite{OMFH11}, allows one to derive various recovery conditions for the Schatten $p$--quasi--norm heuristic from their counterparts for the $\ell_p$ heuristic~\cite{OMFH11}.  Moreover, one can obtain stronger RIP--based recovery guarantees for the Schatten $p$--quasi--norm heuristic~\cite{LLLW12}.  Thus, there is a strong motivation to study Question (Q).  As it turns out, long before the interest in low--rank matrix recovery takes shape, Ando~\cite{A88} has already shown that the perturbation inequality (\ref{eq:key-ineq}) is valid when $A,B$ are positive semidefinite.  This result is later rediscovered by Lai et al.~\cite{LLLW12}.  More recently, Zhang and Qiu~\cite{ZQ10} claimed to have established~(\ref{eq:key-ineq}) in its full generality.  However, as we shall explain in Section~\ref{sec:gap}, there is a critical gap in the proof.\footnote{This is also confirmed by the authors of~\cite{ZQ10} in a private correspondence.}  Thus, to the best of our knowledge, Question~(Q) remains open; see also~\cite[Section 7]{AK12}.


In this paper, we show that the perturbation inequality (\ref{eq:key-ineq}) is indeed valid, thereby giving the first complete answer to Question~(Q).  In fact, we shall prove the following more general result:
\begin{thm} \label{thm:cav-perturb}
Let $A,B \in \R^{m\times n}$ be given matrices.  Suppose that $f:\R_+\limto\R_+$ is a concave function satisfying $f(0)=0$.  Then, we have
\begin{equation}\label{eq:cav-perturb}
\sum_{i=1}^{\min\{m,n\}} \big| f(\sigma_i(A)) - f(\sigma_i(B)) \big| \le \sum_{i=1}^{\min\{m,n\}} f(\sigma_i(A-B)).
\end{equation}
\end{thm}
Since $x\mapsto |x|^p$ is concave on $\R_+$ for any $p \in (0,1]$, by taking $f(\cdot)=(\cdot)^p$ in (\ref{eq:cav-perturb}), we immediately obtain (\ref{eq:key-ineq}).  Our proof of (\ref{eq:cav-perturb}), which is given in Section~\ref{sec:pf}, is inspired in part by the work of Fiedler~\cite{F71} and makes heavy use of matrix perturbation theory.  Then, in Section~\ref{sec:app}, we shall discuss some applications of the perturbation inequality (\ref{eq:key-ineq}) in the study of low--rank matrix recovery.  Finally, we close with some concluding remarks in Section~\ref{sec:concl}.

The following notations will be used throughout this paper.  Let $\mathcal{S}^n$ (resp.~$\mathcal{O}^n$) denote the set of $n\times n$ real symmetric (resp.~orthogonal) matrices.  For an arbitrary matrix $Z \in \R^{m\times n}$, we use $\sigma(Z)$ and $\sigma_i(Z)$ to denote its vector of singular values and $i$--th largest singular value, respectively.  For $Z \in \mathcal{S}^n$, we use $\lambda_i(Z)$ to denote its $i$--th largest eigenvalue.  The spectral norm (i.e., the largest singular value) and Frobenius norm of $Z$ are denoted by $\|Z\|$ and $\|Z\|_F$, respectively.  Given a vector $v$, we use $\mbox{Diag}(v)$ to denote the diagonal matrix with $v$ on the diagonal.  Similarly, given matrices $A_1,\ldots,A_l$, we use $\mbox{BlkDiag}(A_1,\ldots,A_l)$ to denote the block diagonal matrix whose $i$--th diagonal block is $A_i$, for $i=1,\ldots,l$.  We say that $Z=O(\alpha)$ if $\|Z\|/\alpha$ is uniformly bounded as $\alpha\limto0$.

\section{Gap in the Zhang--Qiu Proof} \label{sec:gap}
In this section, we review the main steps in Zhang and Qiu's proof of the perturbation inequality~(\ref{eq:cav-perturb}) and explain the gap in the proof.  To set the stage, let us recall two classic perturbation inequalities:
\begin{enumerate}
\item[\subpb] (Lidskii--Wielandt Eigenvalue Perturbation Inequality) Let $A,B \in \mathcal{S}^l$ be given.  Then, for any $k \in \{1,\ldots,l\}$ and $i_1,\ldots,i_k \in \{1,\ldots,l\}$ satisfying $1\le i_1<\cdots<i_k\le l$,
\begin{equation} \label{eq:lw-ineq}
\sum_{j=1}^k ( \lambda_{i_j}(A) - \lambda_{i_j}(B) ) \le \sum_{i=1}^k \lambda_i(A-B);
\end{equation}
see, e.g.,~\cite[Chapter IV, Theorem 4.8]{SS90}.

\item[\subpb] (Mirsky Singular Value Perturbation Inequality) Let $\bar{A},\bar{B} \in \R^{m\times n}$ be given.  Set $\bar{l} = \min\{m,n\}$.  Then, for any $k \in \{1,\ldots,\bar{l}\}$ and $i_1,\ldots,i_k \in \{1,\ldots,\bar{l}\}$ satisfying $1 \le i_1 < \cdots < i_k \le \bar{l}$,
\begin{equation} \label{eq:mir-ineq}
\sum_{j=1}^k \left| \sigma_{i_j}(\bar{A}) - \sigma_{i_j}(\bar{B}) \right| \le \sum_{i=1}^k \sigma_i( \bar{A}-\bar{B} );
\end{equation}
see, e.g.,~\cite[Chapter IV, Theorem 4.11]{SS90}.
\end{enumerate}
\resetspb
Mirsky~\cite{M60} observed that~(\ref{eq:mir-ineq}) is a simple consequence of~(\ref{eq:lw-ineq}), and his argument goes as follows.  Let
\begin{equation} \label{eq:dilate}
A = \left[ \begin{array}{cc} \bz & \bar{A} \\ \bar{A}^T & \bz \end{array} \right] \in \mathcal{S}^{m+n}, \quad B = \left[ \begin{array}{cc} \bz & \bar{B} \\ \bar{B}^T & \bz \end{array} \right] \in \mathcal{S}^{m+n},
\end{equation}
and suppose without loss of generality that $m\le n$.  It is well--known (see Fact~\ref{fact:sym} below) that $0$ is an eigenvalue of both $A$ and $B$ of multiplicity $n-m$, and the remaining eigenvalues of $A$ and $B$ are $\pm\sigma_1(\bar{A}),\ldots,\pm\sigma_m(\bar{A})$ and $\pm\sigma_1(\bar{B}),\ldots,\pm\sigma_m(\bar{B})$, respectively.  Thus, we have
$$ \left\{ \lambda_i(A) - \lambda_i(B) : i=1,\ldots,m+n \right\} = \left\{ \pm\left| \sigma_i(\bar{A}) - \sigma_i(\bar{B}) \right| : i=1,\ldots,m \right\} \cup \{0\}. $$
In particular, by substituting~(\ref{eq:dilate}) into~(\ref{eq:lw-ineq}), we obtain~(\ref{eq:mir-ineq}).

Motivated by the above argument, Zhang and Qiu first established a Lidskii--Wielandt--type singular value perturbation inequality by extending a matrix--valued triangle inequality of Bourin and Uchiyama~\cite{BU07} and invoking Horn's inequalities for characterizing the eigenvalues of sums of Hermitian matrices~\cite{F00}.  Specifically, they showed that for any concave function $f:\R_+\limto\R_+$ and matrices $A,B \in \R^{m\times n}$, the inequality 
\begin{equation} \label{eq:f-lw-ineq}
   \sum_{j=1}^k \left( f(\sigma_{i_j}(A)) - f(\sigma_{i_j}(B)) \right) \le \sum_{i=1}^k f(\sigma_i(A-B))
\end{equation}
holds for any $k \in \{1,\ldots,\bar{l}\}$ and $i_1,\ldots,i_k \in \{1,\ldots,\bar{l}\}$ satisfying $1 \le i_1<\cdots<i_k \le \bar{l}$, where $\bar{l}=\min\{m,n\}$; cf.~\cite[Theorem 2.1]{ZQ10}.  They then claimed that the perturbation inequality~(\ref{eq:cav-perturb}) follows by applying Mirsky's argument above to~(\ref{eq:f-lw-ineq}); cf.~\cite[Corollary 2.3]{ZQ10}.  However, the reasoning in this last step is flawed.  Indeed, the inequality~(\ref{eq:f-lw-ineq}) is concerned with {\it singular values}, while the inequality~(\ref{eq:lw-ineq}) is concerned with {\it eigenvalues}.  In particular, for the matrices $A,B$ given in~(\ref{eq:dilate}), we only have
$$ \left\{ f(\sigma_i(A)) - f(\sigma_i(B)): i=1,\ldots,m+n \right\} = \left\{ f(\sigma_i(\bar{A})) - f(\sigma_i(\bar{B})):i=1,\ldots,\bar{l} \right\} \cup \{0\}, $$
and there is no guarantee that the set on the right--hand side contains any element of the set 
$$ \left\{ \left| f(\sigma_i(\bar{A})) - f(\sigma_i(\bar{B})) \right| : i=1,\ldots,\bar{l} \right\}. $$
Hence, Mirsky's argument does not lead to the desired conclusion.  In fact, we do not see a straightforward way of proving (\ref{eq:cav-perturb}) using~(\ref{eq:f-lw-ineq}).  The difficulty stems in part from the fact that $f$ is always non--negative, while the eigenvalues in~(\ref{eq:lw-ineq}) can be negative.  This suggests that~(\ref{eq:f-lw-ineq}) is fundamentally different from~(\ref{eq:lw-ineq}).


\section{Proof of the Perturbation Inequality (\ref{eq:cav-perturb})} \label{sec:pf}
In this section, we give the first complete proof of the perturbation inequality~(\ref{eq:cav-perturb}).  The proof can be divided into five steps.  

\medskip
\noindent{\it Step 1: Reduction to the Symmetric Case}

\smallskip
A first observation concerning~(\ref{eq:cav-perturb}) is that we can restrict our attention to the case where both $A$ and $B$ are symmetric.  To prove this, consider the linear operator $\Xi:\mathbb{R}^{m \times n} \limto \mathcal{S}^{m+n}$ given by
$$
\Xi(Z) = \left[
\begin{array}{cc}
\bz & Z \\
Z^T & \bz
\end{array}
\right].
$$
We shall make use of the following standard fact, which establishes a relationship between the singular value decomposition of an arbitrary matrix $Z \in\R^{m \times n}$ and the spectral decomposition of $\Xi(Z) \in \mathcal{S}^{m+n}$:
\begin{fact} \label{fact:sym}
(cf.~\cite[Chapter I, Theorem 4.2]{SS90}) Let $Z \in \R^{m\times n}$ be a given matrix with $m \le n$.  Consider its singular value decomposition $Z = U \left[\begin{array}{cc} \Sigma & \bz \end{array}\right] V^T$, where $U \in \R^{m\times m}$ and $V \in \R^{n\times n}$ are orthogonal and $\Sigma = \mbox{Diag}( \sigma_1(Z),\ldots,\sigma_m(Z) ) \in \mathcal{S}^m$ is diagonal.  Write $V = \left[ \begin{array}{cc} V^1 & V^2 \end{array} \right]$, where $V^1 \in \R^{n \times m}$ and $V^2 \in \R^{n \times (n-m)}$.  Then, the matrix $\Xi(Z)$ admits the spectral decomposition
$$ \Xi(Z) = W \left[ \begin{array}{ccc} \Sigma & \bz & \bz \\ \bz & -\Sigma & \bz \\ \bz & \bz & \bz \end{array} \right] W^T, $$
where
$$ W = \frac{1}{\sqrt{2}} \left[ \begin{array}{ccc} U & U & \bz \\ V^1 & -V^1 & \sqrt{2}\,V^2 \end{array} \right] $$
is orthogonal.  In particular, $0$ is an eigenvalue of $\Xi(Z)$ of multiplicity $n-m$, and the remaining eigenvalues of $\Xi(Z)$ are $\pm \sigma_1(Z),\ldots,\pm \sigma_m(Z)$.  
\end{fact}
Fact \ref{fact:sym} implies that the $i$--th largest singular value of $\Xi(Z)$ is given by
\begin{equation} \label{eq:sym-sv}
\sigma_i(\Xi(Z)) = \left\{
\begin{array}{c@{\quad}l}
   \sigma_{\lceil i/2 \rceil}(Z) & \mbox{for } i=1,\ldots,2m, \\
   \noalign{\medskip}
   0 & \mbox{for } i=2m+1,\ldots,m+n.
\end{array}
\right.
\end{equation}
This in turn implies the following result:
\begin{prop} \label{prop:sym}
The inequality (\ref{eq:cav-perturb}) holds for all matrices $A,B \in \R^{m\times n}$ iff it holds for all symmetric matrices $A,B \in \mathcal{S}^l$.
\end{prop}
\begin{proof}
The ``only if'' part of the proposition is clear.  Suppose then the inequality (\ref{eq:cav-perturb}) holds for all symmetric matrices $A,B \in \mathcal{S}^l$.  Consider arbitrary matrices $A,B \in \R^{m \times n}$, and without loss of generality, suppose that $m \le n$.  By assumption and the linearity of $\Xi$, we have
$$ \sum_{i=1}^{m+n} \big| f(\sigma_i(\Xi(A))) - f(\sigma_i(\Xi(B))) \big| \le \sum_{i=1}^{m+n} f(\sigma_i(\Xi(A-B))). $$
Together with (\ref{eq:sym-sv}), this implies that 
\begin{eqnarray*}
2 \sum_{i=1}^{m} \big| f(\sigma_i(A)) - f(\sigma_i(B)) \big| &=& \sum_{i=1}^{2m} \big| f(\sigma_i(\Xi(A))) - f(\sigma_i(\Xi(B))) \big| \\
\noalign{\medskip}
&\le& \sum_{i=1}^{2m} f(\sigma_i(\Xi(A-B))) \\
\noalign{\medskip}
&=& 2 \sum_{i=1}^{m} f(\sigma_i(A-B)).
\end{eqnarray*}
This completes the proof.
\end{proof}

\medskip
\noindent In view of Proposition \ref{prop:sym}, we will focus on proving~(\ref{eq:cav-perturb}) for the case where $A,B$ are symmetric.  Our strategy is to first establish~(\ref{eq:cav-perturb}) for {\it well--behaved} $f$ (the precise definition will be given shortly).  Then, using a limiting argument, we show that the result can be extended to cover general $f$.

\medskip
\noindent{\it Step 2: Local Behavior of a Well--Behaved $f$}

\smallskip
Let us begin by reviewing some basic facts from convex analysis, as well as introducing some definitions and notations.  By the concavity of $f$, for any $x_l,x_r,y_l,y_r\ge0$ satisfying $x_l < x_r$, $y_l < y_r$, $x_l \le y_l$, and $x_r \le y_r$, we have
\begin{equation} \label{eq:chord-ineq}
\frac{f(x_r)-f(x_l)}{x_r-x_l} \ge \frac{f(y_r)-f(x_l)}{y_r-x_l} \ge \frac{f(y_r)-f(y_l)}{y_r-y_l};
\end{equation}
cf.~\cite[Chapter 5, Lemma 16]{R88}.  This implies that for each $x>0$, the right--hand derivative of $f$ at $x$, which is defined as 
$$ d_f(x) = \lim_{\tau\searrow0} \frac{f(x+\tau)-f(x)}{\tau}, $$
exists and is finite.  Moreover, we have $f(y) \le f(x) + d_f(x)(y-x)$ for any $y\ge0$.  

Now, define the extension $\bar{d}_f:\R_+\limto\R\cup\{+\infty\}$ of $d_f:\R_{++}\limto\R$ by
$$ \bar{d}_f(x) = \left\{ 
\begin{array}{c@{\quad}l}
d_f(x) & \mbox{for } x > 0, \\
\noalign{\medskip}
\displaystyle{ \limsup_{t\searrow0} d_f(t) } & \mbox{for } x=0.
\end{array}
\right.
$$
Using (\ref{eq:chord-ineq}), it can be easily verified that $\bar{d}_f(y) \ge \bar{d}_f(x)$ for all $x \ge y \ge 0$.  We say that $f$ is {\it well--behaved} if $\bar{d}_f(x)<+\infty$ for all $x\ge0$.  Note that for a well--behaved $f$, we have
\begin{equation} \label{eq:supergrad-ineq}
f(y) \le f(x) + \bar{d}_f(x)(y-x)
\end{equation}
for all $x,y\ge0$.

Let $M \in \mathcal{S}^n$ be given.  We say that $\pi=(\pi_1,\ldots,\pi_n)$ is a {\it spectrum--sorting permutation of $M$} if $\pi$ is a permutation of $\{1,\ldots,n\}$ and $\sigma_i(M) = |\lambda_{\pi_i}(M)|$ for $i=1,\ldots,n$.  Note that there can be more than one spectrum--sorting permutation of $M$, as multiple eigenvalues of $M$ can have the same magnitude.  Now, given a spectrum--sorting permutation $\pi$ of $M$, let $M = U \Lambda U^T$ be the spectral decomposition of $M$, where $\Lambda = \mbox{Diag}(\lambda_{\pi_1}(M),\ldots,\lambda_{\pi_n}(M)) \in \mathcal{S}^n$.  Furthermore, define $M_\pi = U \Lambda_\pi U^T$, where $\Lambda_\pi = \mbox{Diag}(s_1,\ldots,s_n) \in \mathcal{S}^n$ and 
$$ s_i = \mbox{sgn}(\lambda_{\pi_i}(M)) \cdot \bar{d}_f(\sigma_i(M)) \quad\mbox{for } i=1,\ldots,n. $$
Our immediate objective is to prove the following theorem, which is the crux of our proof of the perturbation inequality (\ref{eq:cav-perturb}):
\begin{thm} \label{thm:schatten-local}
Let $M,N \in\mathcal{S}^n$ be given.  Suppose that $f$ is well--behaved.  Then, for any spectrum--sorting permutation $\pi$ of $M$ and any scalar $t>0$,
$$ \sum_{i=1}^n f(\sigma_i(M+tN)) \le \sum_{i=1}^n f(\sigma_i(M)) + t \cdot {\rm tr}\left( N M_\pi \right) + O(t^2). $$
\end{thm}
The proof of Theorem \ref{thm:schatten-local} relies on the following fact concerning the singular values of a perturbed symmetric matrix:
\begin{fact} \label{fact:singval-perturb} (cf.~\cite[Section 5.1]{LS05})
Let $M,N \in \mathcal{S}^n$ be given.  Let $i_0,i_1,\ldots,i_l,i_{l+1} \in \{1,\ldots,n+1\}$ be such that $1 = i_0 < i_1<\cdots<i_l<i_{l+1}=n+1$, and for $j=1,\ldots,l$,
\begin{eqnarray}
\sigma_{i_{j-1}}(M) = \cdots = \sigma_{i_j-1}(M) &>& \sigma_{i_j}(M) = \cdots = \sigma_{i_{j+1}-1}(M). \label{eq:singval-not}
\end{eqnarray}
Then, for any $t>0$ and $i \in \{i_j,\ldots,i_{j+1}-1\}$, we have
\begin{equation} \label{eq:sv-perturb}
\sigma_i(M+tN) = \sigma_i(M) + t \cdot \lambda_{i-i_j+1}\left( (Q^j)^T \Xi(N) Q^j \right) + O(t^2),
\end{equation}
where $Q^j$ is a $2n \times (i_{j+1}-i_j)$ matrix whose columns are the eigenvectors associated with the $i_j$--th to the $(i_{j+1}-1)$--st eigenvalue of $\Xi(M)$, for $j=0,1,\ldots,l$.
\end{fact}
{\bf Proof of Theorem \ref{thm:schatten-local}}\,\,\, Using~(\ref{eq:supergrad-ineq}) and~(\ref{eq:sv-perturb}), for $i \in \{i_j,\ldots,i_{j+1}-1\}$ and $j=0,1,\ldots,l$, we have
$$
f(\sigma_i(M+tN)) \le f(\sigma_i(M)) + t \cdot \bar{d}_f(\sigma_i(M)) \cdot \lambda_{i-i_j+1}\left( (Q^j)^T \Xi(N) Q^j \right) + O(t^2).
$$
Hence, 
\begin{eqnarray*}
\sum_{i=1}^n f(\sigma_i(M+tN)) &\le& \sum_{i=1}^n f(\sigma_i(M)) + t \sum_{j=0}^l \sum_{i=i_j}^{i_{j+1}-1} \bar{d}_f(\sigma_i(M)) \cdot \lambda_{i-i_j+1}\left( (Q^j)^T \Xi(N) Q^j \right) + O(t^2) \\
\noalign{\medskip}
&=& \sum_{i=1}^n f(\sigma_i(M)) + t \sum_{j=0}^l \bar{d}_f(\sigma_{i_j}(M)) \cdot \mbox{tr}\left( (Q^j)^T \Xi(N) Q^j \right) + O(t^2),
\end{eqnarray*}
where the last equality follows from (\ref{eq:singval-not}).  Now, fix a spectrum--sorting permutation $\pi$ of $M$.  Let $M=U\Sigma V^T$ be the singular value decomposition of $M$, where $\Sigma=\mbox{Diag}(\sigma_1(M),\ldots,\sigma_n(M)) \in \mathcal{S}^n$.  Here, we take $u_i$ to be the eigenvector corresponding to the eigenvalue $\lambda_{\pi_i}(M)$ and $v_i = \mbox{sgn}(\lambda_{\pi_i}(M))u_i$, where $u_i$ (resp.~$v_i$) is the $i$--th column of $U$ (resp.~$V$), for $i=1,\ldots,n$.  Then, by Fact~\ref{fact:sym}, the matrix $Q^j$ can be put into the form
$$ Q^j = \frac{1}{\sqrt{2}}\left[ \begin{array}{c} U^j \\ V^j \end{array} \right], $$
where $U^j$ (resp.~$V^j$) is the $n \times (i_{j+1}-i_j)$ matrix formed by the $i_j$--th to the $(i_{j+1}-1)$--st column of $U$ (resp.~$V$).  Upon letting
$$ D(M) = \mbox{BlkDiag}\left( \bar{d}_f(\sigma_{i_0}(M))I_{i_1-i_0},\ldots,\bar{d}_f(\sigma_{i_l}(M))I_{i_{l+1}-i_l} \right) \in \R^{n \times n} $$
and noting, because of~(\ref{eq:singval-not}), that $D(M) = \mbox{Diag}(\bar{d}_f(\sigma_1(Z)),\ldots,\bar{d}_f(\sigma_n(Z)))$, we compute
\begin{eqnarray*}
\sum_{j=0}^l \bar{d}_f(\sigma_{i_j}(M)) \cdot \mbox{tr}\left( (Q^j)^T \Xi(N) Q^j \right) &=& \sum_{j=0}^l \bar{d}_f(\sigma_{i_j}(M)) \cdot \mbox{tr} \left( (V^j)^TN(U^j) \right) \\
\noalign{\medskip}
&=& \mbox{tr}\left( NUD(M)V^T \right) \\
\noalign{\medskip}
&=& \mbox{tr}(NM_\pi).
\end{eqnarray*}
This completes the proof. \endproof

\medskip
\noindent{\it Step 3: Lower Bounding the Right--Hand Side of (\ref{eq:cav-perturb}) when $f$ is Well--Behaved}

\smallskip
Let $A,B \in \mathcal{S}^n$ be given, and let $A = U_A \Sigma_A U_A^T$ and $B = U_B \Sigma_B U_B^T$ be the spectral decompositions of $A$ and $B$, respectively.  Set $V = U_A^TU_B \in \mathcal{O}^n$.  We claim that there exists a $Q_0 \in \mathcal{O}^n$ such that
$$ Q_0 = \arg\min_{Q \in \mathcal{O}^n} \sum_{i=1}^n f\left( \sigma_i\left( \Sigma_A - Q\Sigma_BQ^T \right) \right). $$
This follows from the compactness of $\mathcal{O}^n$ and the following result:
\begin{prop}
For each $i \in \{1,\ldots,n\}$, the function $f(\sigma_i(\cdot))$ is continuous on $\mathcal{S}^n$.
\end{prop}
\begin{proof}
Let $i \in \{1,\ldots,n\}$ be fixed.  By~\cite[Chapter IV, Theorem 4.11]{SS90}, $\sigma_i(\cdot)$ is 1--Lipschitz continuous.  Moreover, since $f(\cdot)$ is concave on $\R_+$, it is continuous on $\R_{++}$~\cite[Lemma 2.70]{R06}.  Thus, $f(\sigma_i(\cdot))$ is continuous at all $Z \in \mathcal{S}^n$ satisfying $\sigma_i(Z)>0$.  Now, let $Z \in \mathcal{S}^n$ be such that $\sigma_i(Z)=0$.  Then, using (\ref{eq:supergrad-ineq}) and the fact that $f(0)=0$, we have
$$ \left| f(\sigma_i(Y)) - f(\sigma_i(Z)) \right| \le \left| \bar{d}_f(0) \right| \cdot |\sigma_i(Y)-\sigma_i(Z)| $$
for all $Y \in \mathcal{S}^n$.  This, together with the 1--Lipschitz continuity of $\sigma_i(\cdot)$, implies that $f(\sigma_i(\cdot))$ is continuous at all $Z \in \mathcal{S}^n$ satisfying $\sigma_i(Z)=0$ as well.
\end{proof}

\medskip
As a consequence of the claim, we have
$$
\sum_{i=1}^n f(\sigma_i(A-B)) = \sum_{i=1}^n f\left( \sigma_i \left( \Sigma_A - V\Sigma_BV^T \right) \right) \ge \sum_{i=1}^n f\left( \sigma_i\left( \Sigma_A - Q_0\Sigma_BQ_0^T \right) \right).
$$
We now prove the following result:
\begin{thm} \label{thm:min-comm}
Let $\bar{B} = Q_0\Sigma_B Q_0^T \in \mathcal{S}^n$ and $C=\Sigma_A-\bar{B} \in \mathcal{S}^n$.  Then, $\bar{B}$ and $C$ commute.
\end{thm}
\begin{proof}
Since $\bar{B},C \in \mathcal{S}^n$, we have $\bar{B}$ and $C$ commute iff they are simultaneously diagonalizable.  Moreover, for any spectrum--sorting permutation $\pi$ of $C$, $C_\pi$ has the same set of eigenvectors as $C$.  Thus, $\bar{B}$ and $C$ commute iff $\bar{B}$ and $C_\pi$ commute.  Suppose then that $\bar{B}$ and $C_\pi$ do not commute for some spectrum--sorting permutation $\pi$ of $C$.  Set $D = C_\pi\bar{B} - \bar{B}C_\pi \not=\bz$.  It is easy to verify that $D$ is skew--symmetric, i.e., $D=-D^T$.  Hence, we have $V(t) = \exp(t D) \in \mathcal{O}^n$ for all $t \in \R$.  Since $f$ is well--behaved, we compute
\begin{eqnarray}
\sum_{i=1}^n f\left( \sigma_i\left( \Sigma_A - V(t)\bar{B}V(t)^T \right) \right) &=& \sum_{i=1}^n f\left( \sigma_i\left( \Sigma_A - (I+t D)\bar{B}(I-t D) + O(t^2) \right) \right) \nonumber \\
\noalign{\medskip}
&\le& \sum_{i=1}^n f\left( \sigma_i\left( \Sigma_A - \bar{B} + t(\bar{B}D - D\bar{B}) \right)\right) \nonumber \\
\noalign{\medskip}
&\quad+& \sum_{i=1}^n \left[ \bar{d}_f\left(  \sigma_i\left( \Sigma_A - \bar{B} + t(\bar{B}D - D\bar{B}) \right) \right) \cdot O(t^2) \right] \label{eq:cave-bd} \\
\noalign{\medskip}
&\le& \sum_{i=1}^n f( \sigma_i(C) ) + t\cdot\mbox{tr}\left( (\bar{B}D - D\bar{B})C_\pi \right) + O(t^2), \label{eq:minimizer} 
\end{eqnarray}
where (\ref{eq:cave-bd}) follows from~(\ref{eq:supergrad-ineq}) and the 1--Lipschitz continuity of $\sigma_i(\cdot)$ for $i=1,\ldots,n$, while~(\ref{eq:minimizer}) follows from Theorem~\ref{thm:schatten-local} and the fact that
$$ \bar{d}_f\left(  \sigma_i\left( \Sigma_A - \bar{B} + t(\bar{B}D - D\bar{B}) \right) \right) \le \bar{d}_f(0) < +\infty $$
for all $t \in \R$ and $i \in \{1,\ldots,n\}$.  Using the identity $\mbox{tr}(XY^T)=\mbox{tr}(Y^TX)$, which is valid for arbitrary matrices of the same dimensions, we have
\begin{equation} \label{eq:first-order}
\mbox{tr}\left( (\bar{B}D-D\bar{B})C_\pi \right) = \mbox{tr}\left( -DD^T \right) = -\|D\|_F^2 < 0.
\end{equation}
It follows from (\ref{eq:minimizer}) and (\ref{eq:first-order}) that for sufficiently small $t>0$, 
\begin{equation} \label{eq:contradict}
\sum_{i=1}^n f\left( \sigma_i\left( \Sigma_A - V(t)\bar{B}V(t)^T \right) \right) < \sum_{i=1}^n f( \sigma_i(C) ),
\end{equation}
which contradicts the minimality of $Q_0$.  Hence, we have $D=\bz$, or equivalently, $\bar{B}$ and $C$ commute.
\end{proof}

\medskip
\noindent With the help of the following result, we can gain further insight into the structure of the minimizer $Q_0$.  We omit the proof as it is straightforward.
\begin{prop} \label{prop:diag-comm}
Let $X,Y \in \mathcal{S}^n$ be such that $X$ is diagonal with distinct diagonal entries and $X$ commutes with $Y$.  Then, $Y$ is also diagonal.
\end{prop}
By Theorem \ref{thm:min-comm} and the definition of $C$, we have $\Sigma_A\bar{B}=\bar{B}\Sigma_A$.  If in addition $A$ has distinct eigenvalues, then $\bar{B}$ is diagonal by Proposition \ref{prop:diag-comm}.  In particular, we can write $\bar{B} = \mbox{Diag}(\lambda_{\theta_1}(B),\ldots,\lambda_{\theta_n}(B))$ for some permutation $\theta=(\theta_1,\ldots,\theta_n)$ of $\{1,\ldots,n\}$.  Geometrically, this means that the minimizer $Q_0$ implicitly aligns the principal axes of $A$ and $B$.

\medskip
\noindent{\it Step 4: Proving the Perturbation Inequality~(\ref{eq:cav-perturb}) for Well--Behaved $f$}

\smallskip
For the case where $A \in \mathcal{S}^n$ has distinct eigenvalues, the discussion following Proposition \ref{prop:diag-comm}, together with~\cite[Proposition 1]{AK12}, immediately yields
\begin{eqnarray}
\sum_{i=1}^n f(\sigma_i(A-B)) &\ge& \sum_{i=1}^n f\left( \sigma_i \left( \Sigma_A - Q_0\Sigma_BQ_0^T \right) \right) \nonumber \\
\noalign{\medskip}
&=& \sum_{i=1}^n f\left( \sigma_i \left( \Sigma_A - \mbox{Diag}(\lambda_{\theta_1}(B),\ldots,\lambda_{\theta_n}(B)) \right)\right) \nonumber \\
\noalign{\medskip}
&\ge& \sum_{i=1}^n \big| f(\sigma_i(A)) - f(\sigma_i(B)) \big|. \label{eq:goal-1}
\end{eqnarray}
To handle the case where $A \in\mathcal{S}^n$ has repeated eigenvalues, consider a sequence $\{A^l\}_{l=1}^\infty$ of matrices in $\mathcal{S}^n$ with distinct eigenvalues such that $A^l \limto A$.  By~(\ref{eq:goal-1}), we have
$$ \sum_{i=1}^n f(\sigma_i(A^l-B)) \ge \sum_{i=1}^n \big| f(\sigma_i(A^l)) - f(\sigma_i(B)) \big| $$
for $l=1,2,\ldots$, which by continuity implies that~(\ref{eq:cav-perturb}) holds.

\medskip
\noindent{\it Step 5: Completing the Proof of the Perturbation Inequality~(\ref{eq:cav-perturb})}

\smallskip
To handle the case where $f$ is not well--behaved, we proceed as follows.  For each $\delta>0$, define $f_\delta:\R_+\limto\R_+$ by
$$ f_\delta(x) = \min\left\{ \frac{f(\delta)}{\delta}x,f(x) \right\}. $$
Note that $f_\delta$ is a concave function, as it is the pointwise minimum of two concave functions~\cite[Lemma 2.58]{R06}.  Moreover, since $f(0)=0$, we have $f_\delta(0)=0$.  Thus, $f_\delta$ satisfies the conditions in Theorem~\ref{thm:cav-perturb}.  Now, using the concavity of $f$, it can be shown that
$$ f_\delta(x) = \left\{
\begin{array}{c@{\quad}l}
f(x) & \mbox{for } x \ge \delta, \\
\noalign{\medskip}
\displaystyle{ \frac{f(\delta)}{\delta}x } & \mbox{for } 0 \le x \le \delta.
\end{array}
\right.
$$
In particular, $f_\delta$ is well--behaved.  Thus, by the result in Step~4, we have
$$ \sum_{i=1}^n f_\delta(\sigma_i(A-B)) \ge \sum_{i=1}^n \big| f_\delta(\sigma_i(A)) - f_\delta(\sigma_i(B)) \big| $$
for each $\delta>0$.  To complete the proof, we simply observe that $\lim_{\delta\searrow0}f_\delta(x)\limto f(x)$ for each $x\ge0$.

\section{Applications in Low--Rank Matrix Recovery} \label{sec:app}
As pointed out in~\cite{OMFH11}, one important consequence of the perturbation inequality~(\ref{eq:key-ineq}) is that it connects the sufficient conditions for the recovery of low--rank matrices by the Schatten $p$--quasi--norm heuristic to those for the recovery of sparse vectors by the $\ell_p$ heuristic.  For completeness' sake, let us briefly elaborate on the connection here.

For a given $p \in (0,1)$ and $k \ge 1$, let $\mathscr{S}_k^p$ be the set of $s \times t$ matrices (where $t \ge k$) such that whenever $A \in \mathscr{S}_k^p$, every vector $\bar{x} \in \R^t$ with $\|\bar{x}\|_0 = |\{i:\bar{x}_i\not=0\}| \le k$ and $y = A\bar{x} \in \R^s$ can be exactly recovered by solving the following optimization problem:
\begin{equation} \label{eq:lp}
\begin{array}{c@{\quad}l}
\mbox{minimize} & \|x\|_p^p \\
\noalign{\smallskip}
\mbox{subject to} & Ax = y.
\end{array}
\end{equation}
We have the following theorem:
\begin{thm} \label{thm:transfer}
(cf.~\cite[Theorem 1]{OMFH11}) Let $\mathcal{A}:\R^{m\times n} \limto \R^l$ be a given linear operator with $m \le n$.  Suppose that $\mathcal{A}$ possesses the following property:

\medskip
\noindent{\sc Property (E).} \emph{For any orthogonal $U \in \R^{m \times m}$ and $V \in \R^{n\times n}$, the matrix $A_{U,V} \in \R^{l \times m}$ induced by the linear map $x\mapsto \mathcal{A}\left( U \left[\begin{array}{cc} \mbox{Diag}(x) & \bz \end{array}\right] V^T \right)$ belongs to $\mathscr{S}_k^p$.}

\medskip
\noindent Then, every matrix $\bar{X} \in \R^{m\times n}$ with $\mbox{rank}(\bar{X}) \le k$ and $y = \mathcal{A}(\bar{X}) \in \R^l$ can be exactly recovered by solving Problem (\ref{eq:schatten}) with $\eta = 0$.
\end{thm}
The proof of Theorem \ref{thm:transfer} relies on the following two results, the latter of which is established using the perturbation inequality (\ref{eq:key-ineq}):
\begin{fact} \label{fact:vec-ns}
(cf.~\cite{GN03}) Let $A \in \R^{s \times t}$ be given.  Then, we have $A \in \mathscr{S}_k^p$ iff
$$ \sum_{i=1}^k |z_j^\downarrow|^p < \sum_{i=k+1}^t |z_j^\downarrow|^p \qquad\mbox{for all } z \in \mathcal{N}(A) \backslash \{\bz\}, $$
where $z^\downarrow \in \R^t$ is the vector whose $i$--th entry is the $i$--th largest (in absolute value) entry of $z$, and $\mathcal{N}(A) = \{z \in \R^t: Az = \bz\}$ is the nullspace of $A$.
\end{fact}
\begin{prop} \label{prop:mat-ns}
Let $\mathcal{A}:\R^{m\times n} \limto \R^l$ be a given linear operator with $m\le n$.  Then, every matrix $\bar{X} \in \R^{m\times n}$ with $\mbox{rank}(\bar{X}) \le k$ and $y = \mathcal{A}(\bar{X}) \in \R^l$ can be exactly recovered by solving Problem~(\ref{eq:schatten}) with $\eta = 0$ iff 
\begin{equation} \label{eq:mat-ns}
\sum_{i=1}^k \sigma_i^p(Z) < \sum_{i=k+1}^m \sigma_i^p(Z)
\end{equation}
holds for all $Z \in \mathcal{N}(\mathcal{A}) \backslash \{\bz\}$.
\end{prop}
\begin{proof}
Suppose that (\ref{eq:mat-ns}) holds for all $Z \in \mathcal{N}(\mathcal{A}) \backslash \{\bz\}$.  Let $\bar{X},\bar{X}' \in \R^{m\times n}$ be such that $\mbox{rank}(\bar{X}) \le k$ and $\mathcal{A}(\bar{X}) = \mathcal{A}(\bar{X}') = y$.  Clearly, we have $\bar{Z} = \bar{X}'-\bar{X} \in \mathcal{N}(\mathcal{A})$.  If $\bar{Z} \not= \bz$, or equivalently, if $\bar{X}'\not=\bar{X}$, then by taking $f(\cdot) = (\cdot)^p$ in Theorem \ref{thm:cav-perturb} and using the fact that $\mbox{rank}(\bar{X}) \le k$, we obtain
\begin{eqnarray*}
   \sum_{i=1}^m \sigma_i^p(\bar{X} + \bar{Z}) &\ge& \sum_{i=1}^m \left| \sigma_i^p(\bar{X}) - \sigma_i^p(\bar{Z}) \right| \\
   \noalign{\medskip}
   &=& \sum_{i=1}^k  \left| \sigma_i^p(\bar{X}) - \sigma_i^p(\bar{Z}) \right| + \sum_{i=k+1}^m  \left| \sigma_i^p(\bar{X}) - \sigma_i^p(\bar{Z}) \right| \\
   \noalign{\medskip}
   &\ge& \sum_{i=1}^k \sigma_i^p(\bar{X}) - \sum_{i=1}^k \sigma_i^p(\bar{Z}) + \sum_{i=k+1}^m \sigma_i^p(\bar{Z}) \\
   \noalign{\medskip}
   &>& \sum_{i=1}^m \sigma_i^p(\bar{X}).
\end{eqnarray*}
Since $\bar{X}'=\bar{X}+\bar{Z}$ is arbitrary, this shows that $\bar{X}$ is the unique optimal solution to Problem (\ref{eq:schatten}) when $\eta=0$.

Conversely, suppose there exists a $\bar{Z} \in \mathcal{N}(\mathcal{A}) \backslash \{\bz\}$ such that $\sum_{i=1}^k \sigma_i^p(\bar{Z}) \ge \sum_{i=k+1}^m \sigma_i^p(\bar{Z})$.  Let $\bar{Z} = U \left[\begin{array}{cc} \mbox{Diag}(\sigma(\bar{Z})) & \bz \end{array}\right] V^T$ be its singular value decomposition, and define
$$ \bar{X} = -U \left[\begin{array}{cc} \Sigma_1^k(\bar{Z}) & \bz \end{array}\right] V^T, \quad  \bar{X}' = U \left[\begin{array}{cc} \Sigma_{k+1}^m(\bar{Z}) & \bz \end{array}\right] V^T, $$
where
\begin{eqnarray*}
\Sigma_1^k(\bar{Z}) &=& \mbox{Diag}\left( \sigma_1(\bar{Z}),\ldots,\sigma_k(\bar{Z}),0,\ldots,0 \right) \in \mathcal{S}^m, \\
\noalign{\medskip}
\Sigma_{k+1}^m(\bar{Z}) &=&  \mbox{Diag}\left( 0,\ldots,0,\sigma_{k+1}(\bar{Z}),\ldots,\sigma_m(\bar{Z}) \right) \in \mathcal{S}^m.
\end{eqnarray*}
Clearly, we have $\mbox{rank}(\bar{X}) \le k$.  Moreover, since $\mathcal{A}(\bar{X}'-\bar{X}) = \mathcal{A}(\bar{Z}) = \bz$, we have $\mathcal{A}(\bar{X}) = \mathcal{A}(\bar{X}')$.  Now, using the definition of $\bar{Z}$, we compute
$$ \|\bar{X}\|_p^p = \sum_{i=1}^k \sigma_i^p(\bar{Z}) \ge \sum_{i=k+1}^m \sigma_i^p(\bar{Z}) = \|\bar{X}'\|_p^p. $$
This shows that $\bar{X}$ is not the unique optimal solution to Problem (\ref{eq:schatten}) when $\eta=0$ and $y=\mathcal{A}(\bar{X})$.
\end{proof}

\medskip
\noindent{\bf Proof of Theorem \ref{thm:transfer}}\,\,\, Consider an arbitrary $Z \in \mathcal{N}(\mathcal{A}) \backslash \{\bz\}$.  Let $Z = U \left[\begin{array}{cc} \mbox{Diag}(\sigma(Z)) & \bz \end{array}\right] V^T$ be its singular value decomposition.  Then, we have $\bz = \mathcal{A}(Z) = A_{U,V}(\sigma(Z))$.  Hence, Property~(E) and Fact \ref{fact:vec-ns} imply that $Z$ satisfies~(\ref{eq:mat-ns}).  The desired conclusion now follows from Proposition \ref{prop:mat-ns}. \endproof

\medskip
By invoking existing results in the literature and applying Theorem \ref{thm:transfer}, exact recovery properties of the Schatten $p$--quasi--norm heuristic~(\ref{eq:schatten}) can be derived in a rather straightforward manner.  As an illustration, let us establish two recovery conditions based on notions of restricted isometry for the Schatten $p$--quasi--norm heuristic.  We begin with the following simple observation:
\begin{prop} \label{prop:ext}
Let $m,n,r$ be integers such that $r \le m \le n$.  Let $\mathcal{A}:\R^{m\times n} \limto \R^l$ be a given linear operator.
\begin{enumerate}
\item[\subpb] Suppose there exists a constant $\alpha_r \in (0,1)$ such that
$$ (1-\alpha_r)\|X\|_F^2 \le \|\mathcal{A}(X)\|_2^2 \le (1+\alpha_r) \|X\|_F^2 $$
for all $X \in \R^{m\times n}$ with $\mbox{rank}(X) \le r$.  Then, for any orthogonal $U \in \R^{m\times m}$ and $V \in \R^{n\times n}$, the matrix $A_{U,V} \in \R^{l\times m}$ satisfies
\begin{equation} \label{eq:vec-rip}
(1-\alpha_r)\|x\|_2^2 \le \|A_{U,V}(x)\|_2^2 \le (1+\alpha_r) \|x\|_2^2 
\end{equation}
for all $x \in \R^m$ with $\|x\|_0\le r$.

\item[\subpb] Let $p \in (0,1]$ be given.  Suppose there exists a constant $\beta_{p,r} \in (0,1)$ such that
$$ (1-\beta_{p,r})\|X\|_F^p \le \|\mathcal{A}(X)\|_p^p \le (1+\beta_{p,r}) \|X\|_F^p $$
for all $X \in \R^{m\times n}$ with $\mbox{rank}(X) \le r$.  Then, for any orthogonal $U \in \R^{m\times m}$ and $V \in \R^{n\times n}$, the matrix $A_{U,V} \in \R^{l\times m}$ satisfies
\begin{equation} \label{eq:vec-rpip}
(1-\beta_{p,r})\|x\|_2^p \le \|A_{U,V}(x)\|_p^p \le (1+\beta_{p,r}) \|x\|_2^p
\end{equation}
for all $x \in \R^m$ with $\|x\|_0\le r$.
\end{enumerate}
\resetspb
\end{prop}
\begin{proof}
Let $x \in \R^m$ be such that $\|x\|_0 \le r$.  For any orthogonal $U \in \R^{m\times m}$ and $V \in \R^{n\times n}$, the matrix $X = U \left[\begin{array}{cc} \mbox{Diag}(x) & \bz \end{array}\right] V^T \in \R^{m\times n}$ has rank at most $r$.  Moreover, we have $\|X\|_F = \|x\|_2$, $\|\mathcal{A}(X)\|_2 = \|A_{U,V}(x)\|_2$ and $\|\mathcal{A}(X)\|_p = \|A_{U,V}(x)\|_p$.  This completes the proof.
\end{proof}

\medskip
Condition (\ref{eq:vec-rip}) (resp.~(\ref{eq:vec-rpip})) implies that for any orthogonal $U \in \R^{m\times m}$ and $V \in \R^{n\times n}$, the matrix $A_{U,V} \in \R^{l\times m}$ satisfies the {\it restricted isometry property of order $k$}~\cite{CT05} (resp.~{\it restricted $p$--isometry property of order $k$}~\cite{CS08}) with constant at most $\alpha_r$ (resp.~$\beta_{p,r}$).  Hence, the results in~\cite{CS08,WC13}, together with Theorem \ref{thm:transfer}, imply the following recovery conditions:
\begin{thm}
Let $\mathcal{A}:\R^{m\times n} \limto \R^l$ be a given linear operator with $m \le n$, and let $p \in (0,1)$ be given.
\begin{enumerate}
\item[\subpb] (cf.~\cite{WC13}) Let $k\ge1$ be an integer such that $2k \le m$.  Suppose that $\mathcal{A}$ satisfies the hypothesis of Proposition \ref{prop:ext}(a) with $r=2k$, and that $p < \min\{1,1.0873 \times (1-\alpha_{2k})\}$.  Then, every matrix $\bar{X} \in \R^{m\times n}$ with $\mbox{rank}(\bar{X}) \le k$ and $y = \mathcal{A}(\bar{X}) \in \R^l$ can be exactly recovered by solving Problem (\ref{eq:schatten}) with $\eta = 0$.

\item[\subpb] (cf.~\cite[Theorem 2.4]{CS08}) Given an integer $k\ge1$ and a real number $b>1$, let $a = \lceil b^{2/(2-p)} k \rceil/k$.  Suppose that $\mathcal{A}$ satisfies the hypothesis of Proposition~\ref{prop:ext}(b) with $r=(a+1)k$, and that $\beta_{p,ak}+b\beta_{p,(a+1)k} < b-1$.  Then, every matrix $\bar{X} \in \R^{m\times n}$ with $\mbox{rank}(\bar{X}) \le k$ and $y = \mathcal{A}(\bar{X}) \in \R^l$ can be exactly recovered by solving Problem (\ref{eq:schatten}) with $\eta = 0$.
\end{enumerate}
\resetspb
\end{thm}
For further applications of the perturbation inequality~(\ref{eq:key-ineq}) in the study of low--rank matrix recovery, we refer the reader to~\cite{LLLW12}.

\section{Conclusion} \label{sec:concl}
In this paper, we established the perturbation inequality (\ref{eq:cav-perturb}) concerning the singular values of a matrix.  Such an inequality has proven to be fundamental in understanding the recovery properties of the Schatten $p$--quasi--norm heuristic~(\ref{eq:schatten}).  Thus, a natural future direction is to find other applications of (\ref{eq:cav-perturb}) in the study of low--rank matrix recovery.  Another interesting direction is to prove or disprove the following generalization of~(\ref{eq:cav-perturb}), which has already attracted some attention in the linear algebra community:
\begin{conj}
(\cite[Conjecture 6]{AK12}) Let $A,B \in \R^{m\times n}$ be given.  Suppose that $f:\R_+\limto\R_+$ is a concave function satisfying $f(0)=0$.  Then, for any $k \in \{1,\ldots,\min\{m,n\}\}$,
$$
\sum_{i=1}^k \big| f(\sigma_i(A)) - f(\sigma_i(B)) \big| \le \sum_{i=1}^k f(\sigma_i(A-B)).
$$
\end{conj}





\bibliography{sdpbib}
\bibliographystyle{abbrv}

\end{document}